\documentclass[a4paper, 12pt]{amsart}

\usepackage[
textheight=24.8cm, left=1.5cm, right=1.5cm]{geometry}

\usepackage[T2A]{fontenc}
\usepackage[utf8]{inputenc}
\usepackage[english]{babel}

\usepackage{amsmath,amssymb,amsthm,url}
\usepackage{dsfont,textcomp,graphicx,overpic}
\usepackage[usenames,dvipsnames]{color}
\usepackage{hyperref}
\hypersetup{
   colorlinks   = true, 
   urlcolor     = blue, 
   linkcolor    = blue, 
   citecolor   = red 
}
\usepackage{tikz}
\usetikzlibrary{positioning, calc, intersections, through, arrows}
\usetikzlibrary{graphs}
\usetikzlibrary{graphs.standard}
\usetikzlibrary{patterns}

\newtheoremstyle{mydefinition}
  {\medskipamount}
  {\medskipamount}
  {\normalfont}
  {\parindent}
  {\bfseries}
  {.}
  { }
  {}

\long\def\comment#1\endcomment{}

\def\R{{\mathbb R}}

\newcommand{\arxivonly}[1]{}

\renewcommand{\setminus}[0]{-}

\renewcommand{\labelenumi}{(\alph{enumi})}

\newcommand{\setenumi}{
    \topsep=0mm
    \partopsep=0mm
    \parsep=0mm
    \itemsep=0mm
    \leftmargin=12mm
    \listparindent=0mm
    \itemindent=-5mm
    \labelsep=2mm
    \labelwidth=0mm
    \usecounter{enumi}
    }

\newcommand{\setitemi}{
    \topsep=0mm
    \partopsep=0mm
    \parsep=0mm
    \itemsep=0mm
    \leftmargin=\leftmargini
    \listparindent=2mm
    \itemindent=0mm
    \labelsep=2mm
    \labelwidth=0mm
    }

\renewenvironment{itemize}
    {\begin{list}{$\bullet$}{\setitemi}}
    {\end{list}}

\DeclareMathOperator{\sgn}{sgn}

\usepackage[
backend=biber,
style=alphabetic,
sorting=nty,
maxnames=6,
]{biblatex} 
\newtheorem{theorem}{Theorem}[section]
\newtheorem*{theorem*}{Theorem}
\newtheorem{lemma}[theorem]{Lemma}
\newtheorem{op}[theorem]{Open Problem}

\newtheorem{remark}[theorem]{Remark}

\begin{document}

\title{On winding numbers of almost embeddings of $K_4$ in the plane}

\author[E. Alkin]{Emil Alkin}
\address{Moscow Institute of Physics and Technology, Moscow, Russia}
\email{alkin.ev@phystech.edu}

\author[A. Miroshnikov]{Alexander Miroshnikov}
\address{Moscow Institute of Physics and Technology, Moscow, Russia}
\email{mirosh.av@mail.ru}

\thanks{We are grateful to E. Morozov for suggesting an example \cite[Example~5.4(a)]{ABM+} which allowed us to guess Theorem~\ref{t:main}, although we do not use this example in our proof. We are grateful to A. Skopenkov, and especially to T. Garaev for useful discussions. Some ideas of proof of Theorem~\ref{t:main} were independently invented by A. Lazarev.}


\begin{abstract}
    Let $K_4$ be the complete graph on four vertices.
    Let $f$ be a continuous map of $K_4$ to the plane such that $f$-images of non-adjacent edges are disjoint. For any vertex $v \in K_4$
    take the winding number of the $f$-image of the cycle 
    $K_4 \setminus v$ around $f(v)$. It is known that the sum of these four integers is odd. We 
    construct examples showing
    that this is the only relation between these four numbers.
\end{abstract}

\maketitle

\section{Introduction and the main result}
\label{sec:intro_main}
A classical subject is study of planar graph drawings without self-intersections (i.e. embeddings or plane graphs).   
It is also interesting to study graph drawings having `moderate' self-intersections, e.g.  almost embeddings. 


Let $K_4$ be the complete graph on vertices $\{1, 2, 3, 4\}$. 
A (piecewise-linear\footnote{\label{footnote}We recall this definition in Section~\ref{sec:basic_def}.}, or continuous) map $f:K_4 \to \R^2$ is called an \textbf{almost embedding} if $f(\alpha)\cap f(\beta) = \varnothing$ for any two non-adjacent 
edges $\alpha,\beta\subset K_4$ (see discussion of this definition in Section~\ref{sec:mot}).
We shorten `piecewise-linear almost embedding' to `almost embedding'.


Almost embeddings naturally appear in geometric topology (studies of embeddings), in combinatorial geometry (Helly-type results on convex sets) and in topological combinatorics (topological
Radon and Tverberg theorems). See more motivations in \cite[\S6.10 `Almost embeddings, $\mathbb{Z}_2$- and $\mathbb{Z}$-embeddings']{Algor}, and the references therein. 
For relation to modern research on graphs in the plane see expository papers~\cite{ABM+, ABM+rus}, survey papers~\cite{Sk18, SS, Kyncl},
research papers~\cite{IKN+, Ga23} 
and the references therein.

\medskip





Let $O$ be a point in the plane.
Let $A_1\ldots A_m$ be a closed polygonal line$^{\ref{footnote}}$  not passing through $O$. 
The {\bf winding number} $w(A_1\ldots A_m,O)$ of $A_1\ldots A_m$ around $O$ is the number of revolutions during the rotation of vector whose origin is $O$, and whose endpoint goes along the polygonal line in positive direction.   
Rigorously, 
$$ 2\pi \cdot w(A_1\ldots A_m,O) := 
\angle A_1OA_2+\angle A_2OA_3+\ldots+\angle A_{m-1}OA_m+\angle A_mOA_1$$
is the sum of the oriented angles (this definition is consistent with the well-known definition of the \emph{winding number} of a  \emph{closed continuous curve} in the plane around a given point).

In this paper $K$ is a finite graph. Let $f: K \to \mathbb{R}^2$ be a piecewise-linear map. 
The \emph{restriction} $f|_{ab}$ to an oriented edge $ab$ is the corresponding polygonal line that starts at $f(a)$ and ends at $f(b)$.
Let $C = v_1 \ldots v_n$ be a directed (i.e. oriented) cycle in $K$. 
The \emph{restriction} $f|_C:C\to\R^2$ of $f$ to $C$ is the closed polygonal line $f|_{v_1v_2}\ldots f|_{v_{n-1}v_n}f|_{v_nv_1}$.
For a vertex $v$ in $K$ such that $f(v) \notin f(C)$ denote 
$$w_f(C, v) := w(f|_C ,f(v)).$$

For $j = 1, 2, 3, 4$ denote by $C_j$ the directed cycle in $K_4$ obtained by deleting $j$ from $1234$.
For an almost embedding $f : K_4 \to \mathbb{R}^2$ and a vertex $j$ in $K_4$ denote
$$
w_f(j) := w_f(C_j, j) =  w(f|_{C_j}, f(j)).
$$




\begin{theorem}[folklore] \label{t:radonae}For any continuous almost embedding $f : K_4 \to \R^2$ we have 
$\sum_{j=1}^4 w_f(j) \equiv 1 \pmod 2$. 
\end{theorem}


Our main result (Theorem~\ref{t:main}) is that this is the only relation between these four numbers.
\begin{theorem} \label{t:main}
    For any integers $n_1, n_2, n_3, n_4$ whose sum is odd there is an almost embedding $f: K_4 \to \mathbb{R}^2$ such that $w_f(j) = n_j$ for every $j = 1, 2, 3, 4$.
\end{theorem}

For analogous exercises, see \cite[Problems 6.2.b, 6.3.b]{ABM+} and \cite[Proposition~1.2]{KS20}. For analogous results in high-dimensional space, see \cite[Theorem~1.4]{KS20}, \cite[Theorem~1.3]{Nikkuni}.
Surprisingly, Theorem~\ref{t:k5-edge} is non-analogous result for graphs in the plane.
    


Theorem~\ref{t:radonae} is a corollary\footnote{See a deduction, for example, in \cite[\S5, sketch of a proof of Theorem~5.2]{ABM+}.}, and an analogue in terms of winding numbers, of the famous topological Radon Theorem\footnote{See the formulation in \cite[Theorem~2.2.2.a]{Sk18} or in \cite[Theorem~5.1.2, $d=2$]{Mat}.} for the plane.
The latter is closely related to the famous van Kampen-Flores Theorem\footnote{See the formulation in \cite[Theorem~1.4.1]{Sk18} or in \cite[Theorem~5.1.1, $d=1$]{Mat}; see the relation between them, for example, in \cite[\S4, `Appendix: relatives of the topological Radon theorem']{Sk16}.} for the plane.
Theorem~\ref{t:vankampenae} is a corollary\footnote{See the deduction, for example, in \cite[\S5, sketch of a proof of Theorem~5.5.a]{ABM+}.}, and an analogue in terms of winding numbers, of the van Kampen-Flores Theorem.
\begin{center}
    \begin{tikzpicture}[node distance=1cm, auto]
        \tikzstyle{theorem} = [minimum width=2.5cm, minimum height=1cm, text centered, align=center]

        \node (T1) [theorem] {Theorem~\ref{t:main}};
        \node (T2) [theorem, right of=T1, xshift=3cm] {Theorem~\ref{t:radonae}};
        \node (T3) [theorem, right of=T2, xshift=3.5cm] {Topological Radon \\ Theorem};

        \node (B1) [theorem, below of=T1, yshift=-0.5cm] {Theorem~\ref{t:k5-edge}};
        \node (B2) [theorem, below of=T2, yshift=-0.5cm] {Theorem~\ref{t:vankampenae}};
        \node (B3) [theorem, below of=T3, yshift=-0.5cm] {van Kampen-Flores \\ Theorem};

        \node at ($(T1)!0.5!(T2)$) {$\sim$};
        \node at ($(T2)!0.5!(T3)$) {$\Leftarrow$};
        \node at ($(B1)!0.5!(B2)$) {$\Rightarrow$};
        \node at ($(B2)!0.5!(B3)$) {$\Leftarrow$};
    \end{tikzpicture}
\end{center}
A recent `integer' version (Theorem~\ref{t:k5-edge}) of Theorem~\ref{t:vankampenae} 
suggests an analogous `integer' version of Theorem~\ref{t:radonae}. This `integer' version is obtained from Theorem~\ref{t:radonae} by replacing `$\ \sum_{j=1}^4 w_f(j) \equiv 1 \pmod 2$' with `$\ \sum_{j=1}^4 (-1)^jw_f(j) = \pm 1$' (see detailed explanation in Remark~\ref{rem:radonvankampen}). 
Theorem~\ref{t:main} refutes this suggestion. 

    

\section{Some motivation}
\label{sec:mot}
    This section is not formally used below.

    Denote by $K_5$ the complete graph on vertices $\{1, 2, 3, 4, 5\}$ and by $K_5 \setminus 45$ the graph obtained from $K_5$ by deleting the edge $45$. 
    The definition of an almost embedding of $K_5 \setminus 45$ is analogous to the definition for $K_4$.


    \begin{theorem}[folklore]
    \label{t:vankampenae}
        For any continuous almost embedding $g : K_5 \setminus 45 \to \mathbb{R}^2$ we have
        $$
        w_g(123, 4) - w_g(123, 5) \equiv 1 \pmod 2.
        $$
    \end{theorem} 

    The following non-trivial result (motivated by \cite{KS20}, see detailed motivation in \cite[Remark~2.b]{Ga23}) on winding numbers for a continuous almost embedding of $K_5 - 45$ was proved recently.

    \begin{theorem}[{\cite[Theorem~1]{Ga23}}]
    \label{t:k5-edge}
        For any continuous almost embedding $g : K_5 \setminus 45 \to \mathbb{R}^2$ we have
        $$
        w_g(123, 4) - w_g(123, 5) = \pm 1.
        $$
    \end{theorem} 

    \begin{remark}[relation to the Radon and van Kampen numbers]
    \label{rem:radonvankampen}
        Let $f: K_4 \to \mathbb{R}^2$ and $g: K_5 - 45 \to \mathbb{R}^2$ be general position almost embeddings.

        Let $\overline{g}: K_5 \to \mathbb{R}^2$ be a general position piecewise-linear extension of the map $g$.

         It is known that 
        $$\sum_{j=1}^4 w_f(j) \equiv \rho(f) \pmod 2 \quad \text{ and } \quad w_g(123, 4) - w_g(123, 5) \equiv v(\overline{g}) \pmod 2$$
        where $\rho(f)$ is the Radon number of $f$, and $v(\overline{g})$ is the van Kampen number of $\overline{g}$
        (see definitions and proofs in \cite{Sk18} and \cite[\S5, Sketches of the proofs of Theorem~5.2 and Theorem~5.5.a]{ABM+}).

        There is no `integer' version of the van Kampen-Flores Theorem, see details in \cite[Remark~6 (integer versions of van Kampen Theorem)]{Ga23}.  
        However, Theorem~\ref{t:k5-edge} is an `integer' analogue of the difference $w_g(123, 4) - w_g(123, 5)$ being odd.
    \end{remark}

    The following definition of an almost embedding for arbitrary $K$ is slightly different from (but, in case $K = K_4, K_5 \setminus 45$, is equivalent to) the definition of an almost embedding in Section~\ref{sec:intro_main}. 
    A (piecewise-linear, or continuous) map $f:K\to\R^2$ is called an 
    \textbf{almost embedding} if $f(\alpha)\cap f(\beta) = \varnothing$ for any two non-adjacent simplices (i.e. vertices or edges) $\alpha,\beta\subset K$.

    Our main result (Theorem~\ref{t:main}) and Theorem~\ref{t:k5-edge} are steps towards the  following open problem.

\begin{op}[{\cite[Problem~5.7]{ABM+rus}}]
    Fix a finite graph $K$.
    For an almost embedding $f:K \to \mathbb{R}^2$ consider the collection $w_f(C, v)$ of integers, where $v \in K$ is a vertex, and $C \subset K \setminus v$ is an oriented simple cycle. 
    Describe the collections realizable by almost embeddings $K\to\R^2$.
\end{op}

Theorem~\ref{t:main} gives the answer to the open problem in case of $K = K_4$. 
The open problem for $K$ a proper subgraph of $K_4$ is simple.

\arxivonly{
\footnote{A simpler version of Theorem~\ref{t:main} where `sum is odd' is replaced with `alternating sum is equal to $\pm1$' was proved in \cite[\S5, Example 5.10]{ABM+rus}  (this proof uses only the idea of `finger moves' from Section~$\ref{sec:finger_move}$).
This simpler version gives the answer to the open problem for any proper subgraph $K$ of $K_4$ (maybe there exists a simpler direct way to obtain the answer in case of $K$ being a proper subgraph of $K_4$).}
}

The answer to the open problem in the following interesting cases is unknown:
\begin{itemize}
    \item $K$ is obtained from $K_5$ by deleting an edge (cf. Theorem~\ref{t:k5-edge});
    \item $K$ is obtained from $K_{3, 3}$ by deleting an edge (cf. \cite[\S5, 5.7.b]{ABM+});
    \item $K$ is the graph of a cube or an octahedron (cf. \cite[\S5, 5.8]{ABM+}).
\end{itemize}

\section{Basic definitions}
\label{sec:basic_def}
If you are familiar with the definitions of an oriented angle, a polygonal line and its concatenation, and a piecewise-linear map, you can skip this section.

    
    In the plane let $O, A, B, A_1, \ldots, A_m$ be points.
    
    Assume that $A \neq O$ and $B \neq O$ (but possibly $A = B$). Recall that the oriented (a.k.a.
    directed) angle $\angle AOB$ is the number $t \in (-\pi, \pi]$ such that the vector $\overrightarrow{OB}$ is codirected to the vector obtained from $\overrightarrow{OA}$ by the rotation through $t$. 

    A {\bf polygonal line} $A_1 \ldots A_m$ is a tuple $( A_1, A_2, \ldots , A_m )$ of points in the plane. A
    {\bf closed polygonal line} $A_1 \ldots A_m$ is a tuple $( A_1, A_2, \ldots , A_m, A_1 )$ of points in the plane.
    We also denote by $A_1 \ldots A_m$ the unions $\bigcup_{i = 1}^{m-1} A_i A_{i+1}$ and $\bigcup_{i = 1}^{m-1} A_iA_{i+1} \cup A_mA_1$ of segments for non-closed and closed polygonal lines, respectively.
    
    If the last vertex of a polygonal line $l_1$ coincides with the first vertex of a polygonal line $l_2$, then we denote by $l_1l_2$ the polygonal line which is the concatenation $l_1$, $l_2$.
    For a closed polygonal line $l$ and a positive integer $n$ by $l^n$ we denote the concatenation $\underbrace{l \ldots l}_{n \text{ times}}$.
    For a polygonal line $l$ or a closed polygonal line $l$ by $l^{-1}$ denote the tuple of points of $l$ in reverse order. 
    We shorten $(l^{-1})^n = (l^n)^{-1}$ to $l^{-n}$.
    


    Let $K$ be a graph with $V$ vertices. 
    A {\bf piecewise-linear map} $f:K \to \R^2$ of $K$ to the plane is 
    \begin{itemize}
        \item a collection of $V$ points in the plane corresponding to the vertices, and  
        \item  a collection of (non-closed) polygonal lines in the plane joining those pairs of points from the collection which correspond to pairs of adjacent vertices.
    \end{itemize}

    More precisely, each edge corresponds not to a polygonal line, but to a pair of polygonal lines obtained from each other by passing in the reverse order.
    Setting the orientation on the edge of the graph sets the choice of one of these two polygonal lines.


\section{A reduction of Theorem \ref{t:main} 
}

Theorem~\ref{t:main} is deduced from Lemmas~\ref{l:str_al_em} and \ref{l:n_finger_moves} in this section. 

An almost embedding $f: K_4 \to \mathbb{R}^2$ is said to be \emph{simple} if for every edge $\sigma$ of $K_4$ the polygonal line $f|_\sigma$ does not have self-intersections (i.e. is simple).

 \begin{lemma}
\label{l:str_al_em}
    For any integers $m_1, m_2, m_3$ there exists a simple almost embedding $f: K_4 \to \mathbb{R}^2$ such that 
    $$w_f(1) = m_1, \qquad w_f(2) = -m_2, \qquad w_f(3) = m_3, \qquad w_f(4) = 1 - m_1 - m_2 - m_3.$$
\end{lemma}


\begin{lemma}\label{l:n_finger_moves}
    For any integer $n$ and simple almost embedding $g: K_4 \to \R^2$ there exists an almost embedding $f: K_4 \to \R^2$ such that
    $$w_f(1) = w_g(1), \qquad w_f(2) = w_g(2) + n, \qquad w_f(3) = w_g(3), \qquad w_f(4) = w_g(4) - n.$$
\end{lemma}


\begin{proof}[Deduction of Theorem~\ref{t:main} from Lemmas~\ref{l:str_al_em}~and~\ref{l:n_finger_moves}]
    Since the sum $n_1 + n_2 + n_3 + n_4$ is odd, the numbers $a := (1 - n_1 - n_2 - n_3 - n_4)/2$ and $b := (1 - n_1 + n_2 - n_3 - n_4)/2$ are integers. 
    Using Lemma~\ref{l:str_al_em} for $m_1 := n_1$, $m_2 := a$, $m_3 := n_3$, we obtain a simple almost embedding $g: K_4 \to \mathbb{R}^2$ such that 
    $$w_g(1) = n_1, \qquad w_g(2) = -a = n_2 - b, \qquad w_g(3) = n_3, \qquad w_g(4) = 1 - n_1 - a - n_3 = n_4 + b.$$
    
    To get the required map it suffices to use Lemma~\ref{l:n_finger_moves} for  $g$ and $n := b$.

\end{proof}

\section{Winding number of non-closed polygonal line and its properties}

In the proofs of Lemma~\ref{l:str_al_em} and Lemma~\ref{l:n_finger_moves} we use the following notion.

Let $A_1\ldots A_m$ be a polygonal line not passing through a point $O$. 
Define the real number $w'(A_1\ldots A_m,O)$ by 
$$2\pi \cdot w'(A_1\ldots A_m,O) := \angle A_1OA_2+\angle A_2OA_3+\ldots+\angle A_{m-1}OA_m.$$

Clearly, for any polygonal line $A_1  \ldots A_k \ldots A_m$ not passing through $O$ we have
\begin{equation}\label{eq:add}
    w'(A_1 \ldots A_m,O) = w'(A_1 \ldots A_k,O) + w'(A_k \ldots A_m,O);
    \tag{$*$}
\end{equation}
\begin{equation}\label{eq:inv}
    w'(A_m \ldots A_1,O) = -w'(A_1 \ldots A_m,O); 
    \tag{$**$}
\end{equation}
\begin{equation}\label{eq:conn}
    w'(A_1 \ldots A_{m-1} A_m, O) = w(A_1 \ldots A_{m-1}, O) \  \text{if $A_1 = A_m$}.
    \tag{$**$$*$}
\end{equation}

Let $f:K\to\R^2$ be a piecewise-linear map. For a vertex $v$ and an edge $ij$ in $K$ such that $f(v) \notin f(ij)$ denote 
$$w'_f(ij, v) := w'(f|_{ij} ,f(v)).$$
Then by the properties (\ref{eq:add}) and (\ref{eq:conn}), for any almost embedding $f: K_4 \to \mathbb{R}^2$, and vertices $\{j, a, b, c\} = \{1, 2, 3, 4\}$ of $K_4$ such that $C_j = abc$, we have
\begin{equation}\label{eq:tri}
    w_f(j) = w'_f(ab, j) + w'_f(bc, j) + w'_f(ca, j).
    \tag{$**$$**$}
\end{equation} 


\section{Proof of Lemma~\ref{l:str_al_em}}
First, we define $f(1), f(2)$ and $f(3)$ as vertices of a regular triangle, which are  numbered counterclockwise (i.e. such that $\angle f(2)f(1)f(3) > 0$).
Define $f(4)$ as the center of triangle $f(1)f(2)f(3)$.

Since for any edge of $K_4$ there is only one non-adjacent edge, we construct an almost embedding $f$ independently for those three pairs of non-adjacent edges.


\begin{lemma}
\label{l:construction}
    Let $A$ be a vertex of a regular triangle $ABC$ oriented counterclockwise (i.e. such that $\angle BAC > 0$).
    Denote by $O$ its center.
    For any integer $m$ there exist two disjoint simple polygonal lines $A \ldots O$ and $B \ldots C$ joining $A$ to $O$ and $B$ to $C$, respectively, such that
    $$
    w'(B \ldots C, A) = \frac{1}{6}, \qquad 
    w'(A \ldots O, B) = -\frac{1}{12} + m, \qquad
    w'(A \ldots O, C) = \frac{1}{12}, \qquad 
    w'(B \ldots C, O) = \frac{1}{3} - m.
    $$
\end{lemma}

(See Figure~\ref{f:1} to get an idea of the construction. See the proof of Lemma~\ref{l:construction} at the end of the section.)

\emph{Continuation of the proof of Lemma~\ref{l:str_al_em}.}
Let $\{i, j, k\} := \{1, 2, 3\}$ so that vertices $f(i)$, $f(j)$, $f(k)$ of the triangle $f(123)$ are listed in counterclockwise order.
We define $f|_{i4} = f(i) \ldots f(4)$ and $f|_{jk} = f(j) \ldots f(k)$, where $f(i) \ldots f(4)$ and $f(j) \ldots f(k)$ are polygonal lines from Lemma~\ref{l:construction} applied to a vertex $A := f(i)$ of regular triangle $ABC := f(i)f(j)f(k)$ with center $O := f(4)$, and $m := m_{j}$.

Clearly, the defined $f$ is an almost embedding. Let us check the conditions from the formulation.
$$
w_f(1) 
\overset{\text{(\ref{eq:tri})}}{=}
w'_f(23, 1) + w'_f(34, 1) + w'_f(42, 1) = \frac{1}{6} + \left( - \frac{1}{12} + m_1 \right)  - \frac{1}{12} = m_1,
$$
$$
w_f(2) 
\overset{\text{(\ref{eq:tri})}}{=}
w'_f(13, 2) + w'_f(34, 2) + w'_f(41, 2) = -\frac{1}{6} + \frac{1}{12} + \left( \frac{1}{12} - m_2 \right) = -m_2,
$$
$$
w_f(3) 
\overset{\text{(\ref{eq:tri})}}{=}
w'_f(12, 3) + w'_f(24, 3) + w'_f(41, 3) =  \frac{1}{6} + \left( -\frac{1}{12} + m_3 \right) -\frac{1}{12} = m_3,
$$
$$
w_f(4) 
\overset{\text{(\ref{eq:tri})}}{=}
w'_f(12, 4) + w'_f(23, 4) + w'_f(31, 4) = \left( \frac{1}{3} - m_2 \right) + \left( \frac{1}{3} - m_3 \right) + \left( \frac{1}{3} - m_1 \right) = 1 - m_1 - m_2 - m_3.
$$
\hfill $\square$ 

\begin{figure}[!ht]
    \includegraphics[scale=0.85]{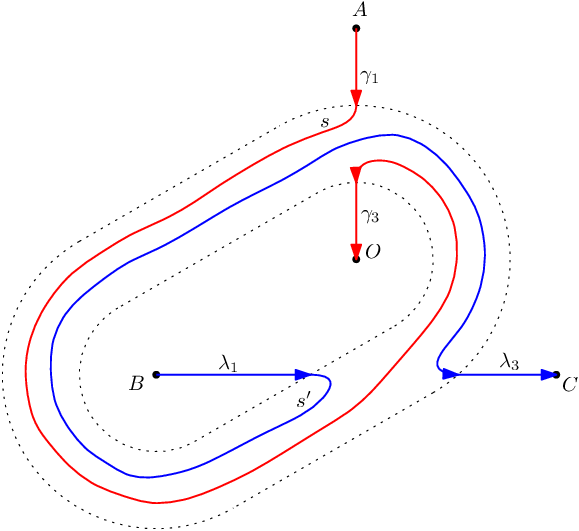}
    \quad
    \includegraphics[scale=0.85]{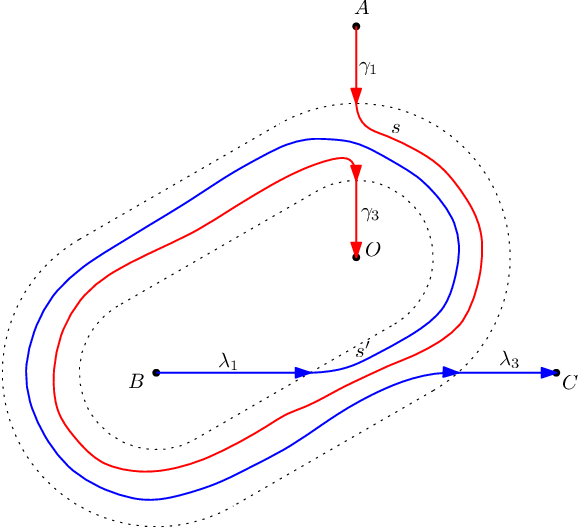}
    \caption{Polygonal lines $A \ldots O$ (red line) and $B \ldots C$ (blue line) for $m = 1$ (left) and for $m = -1$ (right).
    `Spiral' parts of both polygonal lines are depicted as continious lines. 
    Dotted lines refer to the border $\partial R$ for $\varepsilon_1 = \frac{|AO|}{3}$ and $\varepsilon_2 = \frac{2|AO|}{3}$. }
    \label{f:1}
\end{figure}

\begin{proof}[Proof of Lemma~\ref{l:construction}\arxivonly{
(see more explicit proof of Lemma~\ref{l:construction} in the Addendum)
}]
If $m = 0$, then we set $A \ldots O := AO$ and $B \ldots C := BC$. 

Assume $m \neq 0$. Take $0 < \varepsilon_1 < \varepsilon_2 < |AO|$.
Denote $R := \{ p \in \mathbb{R}^2 : \varepsilon_1 \leqslant \rho(p, BO) \leqslant \varepsilon_2 \}$ (this is a subset of $\mathbb{R}^2$ homeomorphic to the annulus $S^1 \times D^1$). 
Clearly, $\partial R$ splits the segment $AO$ into three segments. Denote them by $\gamma_1$, $\gamma_2$, $\gamma_3$ ordered and directed from $A$ to $O$ (see Figure~\ref{f:1}). Analogously, $\partial R$ splits the segment $BC$ into three segments. Denote them by $\lambda_1$, $\lambda_2$, $\lambda_3$ ordered and directed from $B$ to $C$.

Take disjoint simple polygonal lines (spirals) $s$, $s'$ inside $R$ so that

\begin{itemize}
    \item[(a)] $s$ joins the end of $\gamma_1$ to the origin of $\gamma_3$;
    \item[(b)] $s'$ joins the end of $\lambda_1$ to the origin of $\lambda_3$;
    \item[(c)] $s$ makes $|m|$ rotations counterclockwise/clockwise, if $m$ is positive/negative, around the segment~$BO$;
    \item[(d)] $s'$ makes $|m|$ rotations clockwise/counterclockwise, if $m$ is positive/negative, around the segment~$BO$.
\end{itemize}

The conditions (a), (b) imply that $s\gamma_2^{-1}$ and $s'\lambda_2^{-1}$ are closed polygonal lines. 
Since each point $A$ and $C$ can be separated from the set $R$ by a straight line, we have
\begin{equation}
    w(s\gamma_2^{-1}, C) = 0,
    \qquad
    \text{and}
    \qquad
    w(s'\lambda_2^{-1}, A) = 0.
\end{equation}

The conditions (c), (d) imply that
\begin{equation}
    w(s\gamma_2^{-1}, B) = m,
    \qquad
    \text{and}
    \qquad
    w(s'\lambda_2^{-1}, O) = -m.
\end{equation}

We set $A \ldots O := \gamma_1s\gamma_3$ and $B \ldots C := \lambda_1s'\lambda_3$. 

Let us check the conditions from the formulation. 
We have
$$w'(B \ldots C, A) 
\overset{\text{(\ref{eq:add}, \ref{eq:inv})}}{=}
w'(\lambda_1\lambda_2\lambda_3, A) + w'(s'\lambda_2^{-1}, A) 
\overset{\text{(\ref{eq:conn})}}{=}
w'(BC, A) + w(s'\lambda_2^{-1}, A) \overset{(1)}{=} \frac{\angle BAC}{2\pi} = \frac{1}{6},$$

$$w'(A \ldots O, B) 
\overset{\text{(\ref{eq:add}, \ref{eq:inv})}}{=}
w'(\gamma_1\gamma_2\gamma_3, B) + w'(s\gamma_2^{-1}, B) 
\overset{\text{(\ref{eq:conn})}}{=}
w'(AO, B) + w(s\gamma_2^{-1}, B) 
\overset{(2)}{=} 
\frac{\angle ABO}{2\pi} + m =  -\frac{1}{12} + m,$$

$$w'(A \ldots O, C) 
\overset{\text{(\ref{eq:add}, \ref{eq:inv})}}{=}
w'(\gamma_1\gamma_2\gamma_3, C) + w'(s\gamma_2^{-1}, C) 
\overset{\text{(\ref{eq:conn})}}{=}
w'(AO, C) + w(s\gamma_2^{-1}, C) 
\overset{(1)}{=} 
\frac{\angle ACO}{2\pi} = \frac{1}{12},$$

$$w'(B \ldots C, O) 
\overset{\text{(\ref{eq:add}, \ref{eq:inv})}}{=}
w'(\lambda_1\lambda_2\lambda_3, O) + w'(s'\lambda_2^{-1}, O) 
\overset{\text{(\ref{eq:conn})}}{=}
w'(BC, O) + w(s'\lambda_2^{-1}, O) 
\overset{(2)}{=}  
\frac{\angle BOC}{2\pi} - m = \frac{1}{3} - m.$$
\end{proof}


\section{Proof of Lemma~\ref{l:n_finger_moves}}
\label{sec:finger_move}
\textbf{Comment.} \textit{This proof is a formalization of the idea of `finger move' of a polygonal line around a segment (see Figure~\ref{f:finger_move} for illustrating this idea).
Roughly speaking, to obtain the required map we make $|n|$ finger moves (positive or negative) of the polygonal line $g|_{13}$ around the simple polygonal line $g|_{24}$.}

\begin{figure}[ht]
    \includegraphics[scale=0.8]{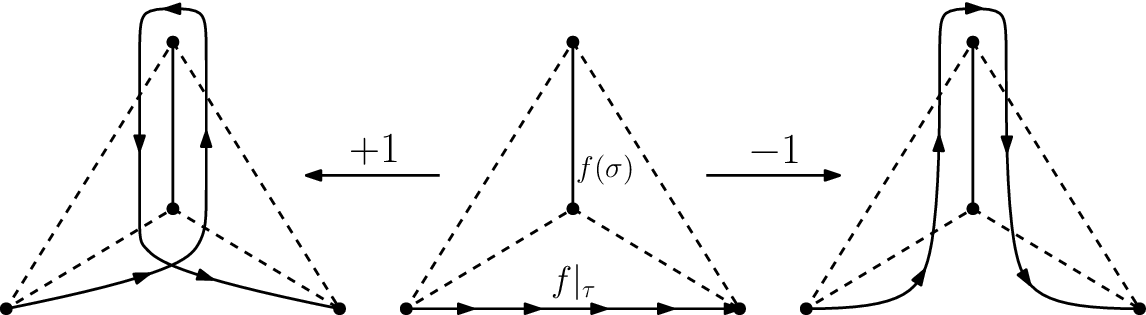}
    \caption{`Finger moves' of a polygonal line $f|_\tau$ around a segment $f(\sigma)$: positive (left) and negative (right)}\label{f:finger_move}
    \label{f:2}
\end{figure}


Since $g$ is a simple almost embedding, there exists a simple closed polygonal line $L$ (see Figure~\ref{f:3}), oriented counterclockwise, such that

$\bullet$ $L$ is disjoint from the polygonal lines $g|_{13}$ and $g|_{24}$, and

$\bullet$ the polygonal line $g|_{24}$ lies in the bounded connected component of $\mathbb{R}^2 \setminus L$, while the polygonal line $g|_{13}$ lies in the unbounded connected component of $\mathbb{R}^2 \setminus L$.
By construction of $L$, 
for any integer $n$ we have
\begin{equation}
    w(L^n, g(2)) 
    \overset{\text{(\ref{eq:conn})}}{=} 
    w'(\underbrace{L^{\sgn n} \ldots L^{\sgn n}}_{|n|\ \text{times}}, g(2)) 
    \overset{\text{(\ref{eq:add})}}{=} 
    |n| \cdot w'(L^{\sgn n}, g(2)) 
    \overset{\text{(\ref{eq:inv})}}{=}
    |n| \cdot \sgn n = n;
\end{equation}
and, analogously,
\begin{equation}
    w(L^n, g(4)) = n.
\end{equation}


\begin{figure}[!ht]
    \includegraphics[scale=0.85]{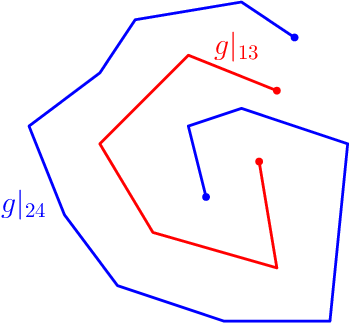}
    \quad
    \includegraphics[scale=0.85]{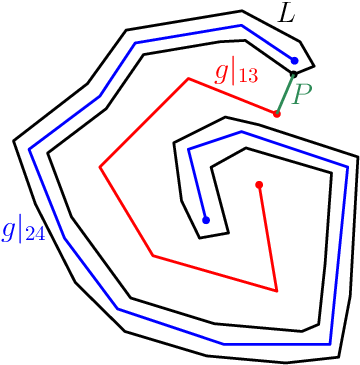}
    \quad
    \includegraphics[scale=0.85]{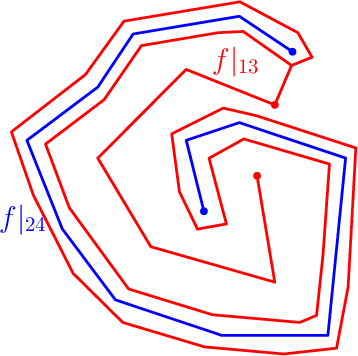}
    \caption{Left: images of edges $13$ and $24$ under the map $g$.
    Middle: $L$ is a simple closed (black) line disjoint from $g|_{13}$ and $g|_{24}$; $P$ is a (green) line disjoint from $g|_{24}$ and joining the point $g(3)$ to a vertex of $L$.
    Right: images of edges $13$ and $24$ under the map $f$.
    }
    \label{f:3}
\end{figure}


Let us define a map $f:K_4 \to \mathbb{R}^2$ as follows: for every directed edge $\sigma \in \{ 12, 14, 23, 24, 34 \}$ in $K_4$ the map $f$ coincides with the map $g$ but $f|_{13} = g|_{13}PL^nP^{-1}$,  where $P$ is a polygonal line disjoint from $g|_{24}$ and joining $g(3)$ to a point of $L$.



The map $f$ is an almost embedding because the map $g$ is an almost embedding and polygonal lines $f|_{13} = g|_{13}PL^nP^{-1}$ and $f|_{24} = g|_{24}$ are disjoint.

Let us check the conditions from the formulation.

Since cycles $C_1$ and $C_3$ do not contain edge $13$, we have $f|_{C_1} = g|_{C_1}$ and $f|_{C_3} = g|_{C_3}$. Hence $w_f(i) = w_g(i)$ for $i \in \{1, 3 \}$. 

We have
$$
w_f(2) 
\overset{\text{(\ref{eq:tri})}}{=}
w'_f(13, 2) + w'_f(34, 2) + w'_f(41, 2) 
\overset{\text{(\ref{eq:add})}}{=} 
$$
$$
\big( w'_g(13, 2) + w'(L^n, g(2)) \big) + w'_g(34, 2) + w'_g(41, 2) 
\overset{\text{(\ref{eq:tri})}}{=}
w_g(2) + w(L^n, g(2)) 
\overset{(3)}{=}
w_g(2) + n,\ \text{and}
$$
$$
w_f(4) 
\overset{\text{(\ref{eq:tri})}}{=}
w'_f(12, 4) + w'_f(23, 4) + w'_f(31, 4) 
\overset{\text{(\ref{eq:add})}}{=}
$$
$$
w'_g(12, 4) + w'_g(23, 4) + \big( w'_g(31, 4) + w(L^{-n}, g(4)) \big) 
\overset{\text{(\ref{eq:tri})}}{=}
w_g(4) + w(L^{-n}, g(4)) 
\overset{(4)}{=}  w_g(4) - n.
$$
$\hfill${$\square$}

\section{Addendum: more explicit proof of Lemma~\ref{l:construction}}

    If $m = 0$, then we set $A \ldots O := AO$ and $B \ldots C := BC$.
    Below we define $A \ldots O$ ($\Gamma$ for short) and $B \ldots C$ ($\Lambda$ for short) for positive $m$. 
    For the negative $m$ swap letters `$Q$' and `$S$' in the construction below.

    Let $T$ be the midpoint of $AC$. 
    Let $S$, $X$, $Q$ be images of points $A$, $
    T$, $C$ under the homothety with center $B$ and ratio $\frac{1}{3}$. Let $R$ be a point such that $X$ is the midpoint of $TR$. Then $QRST$ is the boundary of a rhombus $\Omega$ with center $X$. 
    
    Let $S_1, S_2, \ldots, S_{2m-1}$ be the points dividing $XS$ into $2m$ equal parts numbered from $X$ to $S$. Let $T_1, T_2, \ldots, T_{2m-1}$ be the points dividing $OT$ into $2m$ equal parts numbered from $O$ to $T$. By $'$ we denote a point reflection w.r.t. $X$. For every $i \in \{ 1, \ldots, 2m-1\}$ denote $Q_i := S'_i$ and $R_i := T'_i$.

    Let $Z := OQ_1R_1S_2T_2 \ldots Q_{2m-1}R_{2m-1}ST$ be a polygonal line joining $O$ to $T$ inside the rhombus $\Omega$. Then $Z' = O'Q'_1R'_1S'_2T'_2 \ldots Q'_{2m-1}R'_{2m-1}S'T' = BS_1T_1Q_2R_2 \ldots S_{2m-1}T_{2m-1}QR$ is a polygonal line joining $B$ to $R$ inside the rhombus $\Omega$. We have $Z \cap Z' = \varnothing$, because for any pair $(X, Y) \in \{ (Q, R),\ (R, S),\ (S, T),\ (T, Q) \}$ the intersection $X_iY_j \cap X_kY_l$ is empty if and only if $[i < k] = [j < l]$, where $[\ \cdot\ ]$ is the Iverson bracket. Note that $w'(Z, B) = |m| \cdot ( \angle TXQ + \angle QXR + \angle RXS + \angle SXT) = -m$. 

    Finally, we set $\Gamma := A Z^{-1}$ and $\Lambda := Z'C$ (see Figure~\ref{f:4}). Since $Z \subset \Omega \setminus \{ R \}$, then $Z \cap RC = \varnothing$ and $Z' \cap AT = \varnothing$. Hence $\Gamma \cap \Lambda = (Z \cap Z') \cup (Z \cap RC) \cup (AT \cap Z') \cup (AT \cap RC) = \varnothing$.

    Let us check the conditions from the formulation of the lemma: 
    $$
    w'(\Lambda, A) = \angle BAC = \frac{1}{6};
    $$
    $$
    w'(\Gamma, B) = w'(AT, B) + w'(Z^{-1}, B) = -\frac{1}{12} - w'(Z, B) = -\frac{1}{12} + m;
    $$
    $$
    w'(\Gamma, C) = \angle ACO = \frac{1}{12};
    $$
    $$
    w'(\Lambda, O) = w'(Z', O) + w'(RC, O) = w'(Z, B) + \frac{1}{3} = -m + \frac{1}{3}.
    $$


    \begin{figure}[!htb]
    \includegraphics[scale=1.3]{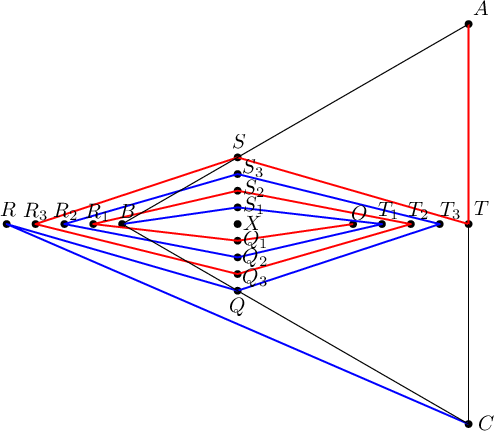}
    \caption{
    Polygonal lines $A \ldots O$ (red line) and $B \ldots C$ (blue line) for $m = 2$
    } 
    \label{f:4}
    \end{figure}












\printbibliography

\end{document}